\documentclass{article}
\usepackage{amsfonts}
\usepackage{mathrsfs}
\usepackage{amsthm}
\usepackage{amssymb}
\usepackage{amsmath}
\usepackage{enumerate}

\theoremstyle{plain}
\newtheorem{theorem}{Theorem}[section]

\newtheorem{corollary}[theorem]{Corollary}
\newtheorem{lemma}[theorem]{Lemma}

\theoremstyle{definition}
\newtheorem{definition}{Definition}[section]

\theoremstyle{remark}
\newtheorem{remark}{\textbf{Remark}}[section]

\theoremstyle{example}
\newtheorem{example}{Example}[section]

\numberwithin{equation}{section}

\title{Higher-Dimensional Open Quantum Walk Constructed from Quantum Bernoulli Noises}
\author{Ce Wang\\
School of Mathematics and Statistics \& Center for Mathematical Sciences\\
Huazhong University of Science and Technology\\
Wuhan 430074, People's Republic of China}
\date{}
\begin{document}
\maketitle

\noindent\textbf{Abstract.}\ \
Quantum Bernoulli noises are annihilation and creation operators acting on Bernoulli functionals,
which satisfy the canonical anti-commutation relations (CAR) in equal-time.
In this paper, we use quantum Bernoulli noises to introduce a model of open quantum walk on the $d$-dimensional
integer lattice $\mathbb{Z}^d$ for a general positive integer $d\geq 2$, which we call the $d$-dimensional open QBN walk.
We obtain a quantum channel representation of the $d$-dimensional open QBN walk,
and find that it admits the ``separability-preserving'' property.
We prove that, for a wide range of choices of its initial state,
the $d$-dimensional open QBN walk has a limit probability distribution of $d$-dimensional Gauss type.
Finally we unveil links between the $d$-dimensional open QBN walk and the unitary quantum walk recently introduced in
[Ce Wang and Caishi Wang, Higher-dimensional quantum walk in terms of quantum Bernoulli noises, Entropy 2020, 22, 504].
\vskip 2mm

\vskip 2mm


\section{Introduction}

Open quantum walks (OQW for short) are also known as open quantum random walks,
which are quantum analogs of classical Markov chains in probability theory.
As a new type of quantum walks, OQWs are finding application in the generalisations of the theory of quantum probability,
and have potential application in dissipative quantum computation, quantum state engineering
and transport in mesoscopic systems \cite{sinayskiy}.

Unlike unitary quantum walks (UQW for short), which have been well studied (see \cite{kemp,konno-1, venegas} and references therein),
the dynamics of OQWs are non-unitary due to the effects of the local environments.
Indeed, OQWs, step by step, describe typically quantum behaviors, but seem to show up a rather classical asymptotic behavior \cite{attal-2}.

The first model of OQW was introduced by Attal, Petruccione, Sabot and Sinayskiy
in 2012 (see \cite{attal-1} for details). Since then, much attention has been paid to OQWs.
Attal \textit{et al}  \cite{attal-2} established the central limit theorem (CLT)
for a class of homogeneous OQWs on $\mathbb{Z}^d$ with a unique invariant state.
Konno and Yoo  \cite{konno-2} applied the CLT to the study of limit probability distributions for various OQWs.
Sadowski and Pawela \cite{sadowsk} investigated a generalisation of the CLT for the case of nonhomogenous OQWs.
Sinayskiy and Petruccione examined the properties of OQWs on the $1$-dimensional integer lattice $\mathbb{Z}$ for the case
of simultaneously diagonalisable transition operators \cite{sinayskiy-0}.
Carbone and Pautrat \cite{carbone}, from a perspective of classical Markov chain,
introduced notions of irreducibility, period, communicating classes for OQWs.
Lardizabal \cite{lardizabal} defined a notion of hitting time for OQWs and obtained some useful
formulas for certain cases. There are many other researches on OQWs
(see the recent survey article \cite{sinayskiy} by Sinayskiy and Petruccione and references therein).

From a physical point of view, OQWs are quantum walks where the transitions between the sites
(or vertices) are driven by the interaction with an environment, which can
cause dissipation and decoherence. This suggests that the effects of an environment play an important role
in the time evolution of the involved OQW.
On the other hand, quantum Bernoulli noises \cite{w-c-l} have turned out to be an alternative approach to the environment of
an open quantum system (see, e.g. \cite{w-t-r} and references therein).
It is then natural to apply quantum Bernoulli noises to the study of OQWs.

In 2018, by using quantum Bernoulli noises, Wang \textit{et al} \cite{w-w-r-t} introduced
a model of OQW on the $1$-dimensional integer lattice $\mathbb{Z}$,
which we call \textbf{the $1$-dimensional open QBN walk} below. In this paper,
we would like to extend the $1$-dimensional open QBN walk to a higher-dimensional case.
More specifically, for a general integer $d\geq 2$, we will introduce a model of OQW
on the $d$-dimensional integer lattice $\mathbb{Z}^d$ in terms of quantum Bernoulli noises
and examine its dynamical behavior from a perspective of probability distribution.
Our main work is as follows. Let $\mathcal{H}$ be the space of square integrable Bernoulli functionals and
$\mathscr{T}^{(d)}(\mathcal{H})$ the set of $d$-dimensional nucleuses on $\mathcal{H}$
(see Definition~\ref{def-2-2} for its exact meaning).
\begin{itemize}
  \item By using quantum Bernoulli noises, we construct a sequence of mappings
        $\mathfrak{J}_n^{(d)}$, $n\geq 0$, on $\mathscr{T}^{(d)}(\mathcal{H})$
         to describe the change in the internal degree of freedom of the walker.
  \item With the above-mentioned mappings $\mathfrak{J}_n^{(d)}$, $n\geq 0$, as the main tool, we establish our model of OQW
        on the $d$-dimensional integer lattice $\mathbb{Z}^d$, which we call \textbf{the $d$-dimensional open QBN walk}.
  \item We obtain a quantum channel representation of the $d$-dimensional open QBN walk and find
        that it admits the ``separability-preserving'' property.
  \item We prove that, for a wide range of choices of its initial state, the $d$-dimensional open QBN walk has
        a limit probability distribution of $d$-dimensional Gauss type.
  \item We unveil links between the $d$-dimensional open QBN walk and the unitary quantum walk recently introduced in \cite{w-w},
        which was called the $d$-dimensional QBN walk therein.
\end{itemize}
Some other interesting results are also proven.

The paper is organized as follows. In Subsection~\ref{subsec-2-1}, we briefly recall some necessary notions
and facts about quantum Bernoulli noises.
Subsections~\ref{subsec-2-2}, \ref{subsec-2-3} and \ref{subsec-2-4} are one part of our main work,
which includes several technical theorems we prove, the definition of our model of OQW (namely the $d$-dimensional open QBN walk)
and its quantum channel representation, among others.
Another part of our main work lies in Sections~\ref{sec-3} and~\ref{sec-4}, where we show that
the walk admits the ``separability-preserving'' property,
we prove that, for a wide range of choices of its initial state, the walk has
a limit probability distribution of $d$-dimensional Gauss type,
and finally we show links between the $d$-dimensional open QBN walk and the unitary quantum walk recently introduced in \cite{w-w}.

Throughout this paper, $\mathbb{Z}$ always denotes the set of all integers, while
$\mathbb{N}$ means the set of all nonnegative integers. For a positive integer $d\geq 2$, we use $\mathbb{Z}^d$ to mean
the $d$-dimensional integer lattice.
We denote by $\Gamma$ the finite power set of $\mathbb{N}$, namely
\begin{equation}\label{eq-1-1}
    \Gamma
    = \{\,\sigma \mid \text{$\sigma \subset \mathbb{N}$ and $\#\,\sigma < \infty$} \,\},
\end{equation}
where $\#\sigma$ means the cardinality of $\sigma$.
If $\mathcal{X}$ is a Hilbert space and $d\geq 2$ a positive integer, then $\mathcal{X}^{\otimes n}$ denotes the $d$-fold tensor product of $\mathcal{X}$. We denote by $\mathfrak{B}(\mathcal{X})$ the set of all bounded linear operators on $\mathcal{X}$.
As usual, a density operator on a Hilbert space means a positive operator of trace class with unit trace on that space.
By convention, $\mathrm{Tr}A$ denotes the trace of an operator of trace class $A$.
Unless otherwise stated, letters like $j$, $k$ and $n$ stand for nonnegative integers, namely elements of $\mathbb{N}$.

\section{Definition of walk and its basic properties}\label{sec-2}

In this section, we first define our model of OQW (open quantum walk) on the $d$-dimensional integer lattice $\mathbb{Z}^d$
by using quantum Bernoulli noises, which we will call the $d$-dimensional open QBN walk,
and then we examine its basic properties.

\subsection{Quantum Bernoulli noises}\label{subsec-2-1}

We first briefly recall some necessary notions and facts about quantum Bernoulli noises.
We refer to \cite{w-c-l} for details about quantum Bernoulli noises.

Let $\Omega$ be the set of all functions $f\colon \mathbb{N} \mapsto \{-1,1\}$, and
$(\zeta_n)_{n\geq 0}$ the sequence of canonical projections on $\Omega$ given by
\begin{equation}\label{eq-2-1}
    \zeta_n(f)=f(n),\quad f\in \Omega.
\end{equation}
Let $\mathscr{F}$ be the $\sigma$-field on $\Omega$ generated by the sequence $(\zeta_n)_{n\geq 0}$,
and $(p_n)_{n\geq 0}$ a given sequence of positive numbers with the property that $0 < p_n < 1$ for all $n\geq 0$.
Then there exists a unique probability measure $\mathbb{P}$ on $\mathscr{F}$ such that
\begin{equation}\label{eq-2-2}
\mathbb{P}\circ(\zeta_{n_1}, \zeta_{n_2}, \cdots, \zeta_{n_k})^{-1}\big\{(\epsilon_1, \epsilon_2, \cdots, \epsilon_k)\big\}
=\prod_{j=1}^k p_j^{\frac{1+\epsilon_j}{2}}(1-p_j)^{\frac{1-\epsilon_j}{2}}
\end{equation}
for $n_j\in \mathbb{N}$, $\epsilon_j\in \{-1,1\}$ ($1\leq j \leq k$) with $n_i\neq n_j$ when $i\neq j$
and $k\in \mathbb{N}$ with $k\geq 1$. Thus one has a probability measure space $(\Omega, \mathscr{F}, \mathbb{P})$,
which is referred to as the Bernoulli space and random variables on it are known as Bernoulli functionals.

Let $Z=(Z_n)_{n\geq 0}$ be the sequence of Bernoulli functionals generated by sequence $(\zeta_n)_{n\geq 0}$, namely
\begin{equation}\label{eq-2-3}
   Z_n = \frac{\zeta_n + q_n - p_n}{2\sqrt{p_nq_n}},\quad n\geq0,
\end{equation}
where $q_n = 1-p_n$. Clearly $Z=(Z_n)_{n\geq 0}$ is an independent sequence of random variables on the
probability measure space $(\Omega, \mathscr{F}, \mathbb{P})$.
Let $\mathcal{H}$ be the space of square integrable complex-valued Bernoulli functionals, namely
\begin{equation}\label{eq-2-4}
  \mathcal{H} = L^2(\Omega, \mathscr{F}, \mathbb{P}).
\end{equation}
We denote by
$\langle\cdot,\cdot\rangle$ the usual inner product of the space $\mathcal{H}$, and by $\|\cdot\|$ the corresponding norm.
It is known that $Z$ has the chaotic representation property.
Thus $\mathfrak{Z} = \{Z_{\sigma}\mid \sigma \in \Gamma\}$  form an orthonormal basis (ONB) of $\mathcal{H}$,
which is known as the canonical ONB of $\mathcal{H}$.
Here $Z_{\emptyset}=1$ and
\begin{equation}\label{eq-2-5}
    Z_{\sigma} = \prod_{j\in \sigma}Z_j,\quad \text{$\sigma \in \Gamma$, $\sigma \neq \emptyset$}.
\end{equation}
Clearly $\mathcal{H}$ is infinite-dimensional as a complex Hilbert space.

It can be shown that \cite{w-c-l}, for each $k\in \mathbb{N}$, there exists a bounded operator $\partial_k$ on
$\mathcal{H}$ such that
\begin{equation}\label{eq-2-6}
    \partial_k Z_{\sigma} = \mathbf{1}_{\sigma}(k)Z_{\sigma\setminus k},\quad \partial_k^{\ast} Z_{\sigma}
    = [1-\mathbf{1}_{\sigma}(k)]Z_{\sigma\cup k}\quad
    \sigma \in \Gamma,
    \sigma \in \Gamma,
\end{equation}
where $\partial_k^{\ast}$ denotes the adjoint of $\partial_k$, $\sigma\setminus k=\sigma\setminus \{k\}$, $\sigma\cup k=\sigma\cup \{k\}$
and $\mathbf{1}_{\sigma}(k)$ the indicator of $\sigma$ as a subset of $\mathbb{N}$.

The operators $\partial_k$ and $\partial_k^{\ast}$ are usually known as the annihilation and
creation operators acting on Bernoulli functionals, respectively.
And the family $\{\partial_k, \partial_k^{\ast}\}_{k \geq 0}$ is referred to as \textbf{quantum Bernoulli noises}.

A typical property of quantum Bernoulli noises is that they satisfy
the canonical anti-commutation relations (CAR) in equal-time \cite{w-c-l}.
More specifically, for $k$, $l\in \mathbb{N}$, it holds true that
\begin{equation}\label{eq-2-7}
    \partial_k \partial_l = \partial_l\partial_k,\quad
    \partial_k^{\ast} \partial_l^{\ast} = \partial_l^{\ast}\partial_k^{\ast},\quad
    \partial_k^{\ast} \partial_l = \partial_l\partial_k^{\ast}\quad (k\neq l)
\end{equation}
and
\begin{equation}\label{eq-2-8}
   \partial_k\partial_k= \partial_k^{\ast}\partial_k^{\ast}=0,\quad
   \partial_k\partial_k^{\ast} + \partial_k^{\ast}\partial_k=I,
\end{equation}
where $I$ is the identity operator on $\mathcal{H}$.

For a nonnegative integer $n\geq 0$, one can define, respectively, two self-adjoint operators $L_n$ and $R_n$ on $\mathcal{H}$ in the following manner
\begin{equation}\label{eq-2-9}
L_n = \frac{1}{2}(\partial_n^* + \partial_n -I),\quad R_n=\frac{1}{2}(\partial_n^* + \partial_n +I).
\end{equation}
It then follows from the properties of quantum Bernoulli noises that the operators $L_n$, $R_n$, $n\geq 0$, form a commutative family, namely
\begin{equation}\label{eq-2-10}
  L_kL_l = L_lL_k,\quad R_kL_l = L_lR_k,\quad R_kR_l = R_lR_k,\quad k,\, l\geq 0.
\end{equation}

\begin{lemma}\label{lem-2-3}\cite{w-w-r-t}
For all $n\geq 0$, operators $L_n$ and $R_n$ admit the following operational properties
\begin{equation}\label{eq-2-11}
L_n^2=-L_n,\quad L_nR_n=R_n L_n=0,\quad R_n^2=R_n,\quad L_n^2 + R_n^2=I.
\end{equation}
\end{lemma}

\subsection{Technical theorems}\label{subsec-2-2}

In this subsection, we prove some technical theorems, which will be used in defining our model of OQW
and examining its properties.

In what follows, we always assume that $d\geq 2$ is a given positive integer and $\Lambda =\{-1,\, +1\}$.
We denote by $\Lambda^d$ the $d$-fold cartesian product of $\Lambda$,
and by $\mathcal{H}^{\otimes d}$ the  $d$-fold tensor product space of $\mathcal{H}$.
In addition, we assume that $\mathsf{K}\colon \mathcal{H}^{\otimes d}\rightarrow \mathcal{H}$ is a fixed unitary isomorphism.
Such a unitary isomorphism does exist because $\mathcal{H}$ is infinite-dimensional and separable.

To facilitate our discussions, we further write $\mathfrak{S}(\mathcal{H})$ for the space of all operators
of trace class on $\mathcal{H}$ with the trace norm and $\mathfrak{S}_+(\mathcal{H})$ for the cone of all positive elements of $\mathfrak{S}(\mathcal{H})$.
Similarly, we use symbol $\mathfrak{S}\big(\mathcal{H}^{\otimes d}\big)$ and $\mathfrak{S}_+\big(\mathcal{H}^{\otimes d}\big)$.

\begin{definition}\label{def-2-1}
For $n\geq 0$ and $\varepsilon=(\varepsilon_1, \varepsilon_2, \cdots, \varepsilon_d)\in \Lambda^d$, we define
\begin{equation}\label{eq-2-12}
  C_n^{(\varepsilon)} = \mathsf{K}\Big(\bigotimes_{j=1}^d B_n^{(\varepsilon_j)}\Big)\mathsf{K}^{-1},
\end{equation}
where $\mathsf{K}^{-1}$ is the inverse of the unitary isomorphism $\mathsf{K}\colon \mathcal{H}^{\otimes d}\rightarrow \mathcal{H}$,
and $B_n^{(\varepsilon_j)}$ is given by
\begin{equation}\label{eq-2-13}
  B_n^{(\varepsilon_j)}=
  \left\{
    \begin{array}{ll}
      L_n, & \hbox{$\varepsilon_j =-1$;} \\
      R_n, & \hbox{$\varepsilon_j =+1$}
    \end{array}
  \right.
\end{equation}
for $i=1$, $2$, $\cdots$, $d$.
\end{definition}

It can be shown that, for all $n\geq 0$, $\big\{C_n^{(\varepsilon)}\mid \varepsilon\in \Lambda^d\big\}$ are self-adjoint
operators on $\mathcal{H}$. And moreover, they admit the following useful properties:\ $C_n^{(\varepsilon)}C_n^{(\varepsilon')}=0$
for $\varepsilon$, $\varepsilon'\in \Lambda^d$ with $\varepsilon\neq \varepsilon'$; and their sum
$\sum_{\varepsilon\in \Lambda^d}C_n^{(\varepsilon)}$ is a unitary operator on $\mathcal{H}$ (see \cite{w-w} for details).

\begin{theorem}\label{thr-2-4}
Let $n\geq 0$. Then $\sum_{\varepsilon\in \Lambda^d}C_n^{(\varepsilon)}C_n^{(\varepsilon)}=I$, where $I$ denotes the identity operator on $\mathcal{H}$.
\end{theorem}

\begin{proof}
By the definition of $B_n^{-1}$ and $B_n^{(+1)}$ and Lemma~\ref{lem-2-3}, we have
\begin{equation*}
  \sum_{\varepsilon_j\in \Lambda}B_n^{(\varepsilon_j)}B_n^{(\varepsilon_j)}
  = B_n^{(-1)}B_n^{(-1)} + B_n^{(+1)}B_n^{(+1)}
  = L_n^2 + R_n^2
  = I
\end{equation*}
for $j=1$, $2$, $\cdots$, $d$, where $I$ denotes the identity operator on $\mathcal{H}$.
Making tensor products gives
\begin{equation*}
  \sum_{\varepsilon\in \Lambda^d}\Big(\bigotimes_{j=1}^d B_n^{(\varepsilon_j)}\Big)\Big(\bigotimes_{j=1}^d B_n^{(\varepsilon_j)}\Big)
  = \sum_{\varepsilon\in \Lambda^d}\Big(\bigotimes_{j=1}^d B_n^{(\varepsilon_j)}B_n^{(\varepsilon_j)}\Big)
  = \bigotimes_{j=1}^d\sum_{\varepsilon_j\in \Lambda}B_n^{(\varepsilon_j)}B_n^{(\varepsilon_j)}
  = \bigotimes_{j=1}^dI.
\end{equation*}
This, together with the definition of $C_n^{(\varepsilon)}$ as well as properties of the unitary isomorphism $\mathsf{K}$, yields
\begin{equation*}
  \sum_{\varepsilon\in \Lambda^d}C_n^{(\varepsilon)}C_n^{(\varepsilon)}
  = \mathsf{K}\Big[\sum_{\varepsilon\in \Lambda^d}\Big(\bigotimes_{j=1}^d B_n^{(\varepsilon_j)}\Big)
    \Big(\bigotimes_{j=1}^d B_n^{(\varepsilon_j)}\Big)\Big]\mathsf{K}^{-1}
  = \mathsf{K}\Big[\bigotimes_{j=1}^dI\Big]\mathsf{K}^{-1}
  = I.
\end{equation*}
Here we note that $\bigotimes_{j=1}^dI$ is just the identity operator on $\mathcal{H}^{\otimes d}$.
\end{proof}

\begin{theorem}\label{thr-2-5}
For all $\varrho \in \mathfrak{S}_+(\mathcal{H})$ and $n\geq 0$, the sum operator
$\sum_{\varepsilon\in \Lambda^d}C_n^{(\varepsilon)}\varrho\, C_n^{(\varepsilon)}$ belongs to $\mathfrak{S}_+(\mathcal{H})$,
and moreover it holds true that
\begin{equation}\label{eq-2-14}
  \mathrm{Tr}\Big[\sum_{\varepsilon\in \Lambda^d}C_n^{(\varepsilon)}\varrho\, C_n^{(\varepsilon)}\Big]
  = \mathrm{Tr}\varrho.
\end{equation}
\end{theorem}

\begin{proof}
For each $\varepsilon \in \Lambda^d$, $C_n^{(\varepsilon)}\varrho\, C_n^{(\varepsilon)}$ is a positive
operator of trace class on $\mathcal{H}$ since $\varrho$ is such an operator and $C_n^{(\varepsilon)}$ is self-adjoint.
Thus $\sum_{\varepsilon\in \Lambda^d}C_n^{(\varepsilon)}\varrho\, C_n^{(\varepsilon)}$ is also a
positive operator of trace class on $\mathcal{H}$, namely it belongs to $\mathfrak{S}_+(\mathcal{H})$.
Now, by using Theorem~\ref{thr-2-4}, we find
\begin{equation*}
  \mathrm{Tr}\Big[\sum_{\varepsilon\in \Lambda^d}C_n^{(\varepsilon)}\varrho\, C_n^{(\varepsilon)}\Big]
  = \sum_{\varepsilon\in \Lambda^d} \mathrm{Tr}\big[\varrho\,C_n^{(\varepsilon)} C_n^{(\varepsilon)}\big]
  = \mathrm{Tr}\big[ \varrho \sum_{\varepsilon\in \Lambda^d}C_n^{(\varepsilon)} C_n^{(\varepsilon)}\big]
  =\mathrm{Tr}\varrho.
\end{equation*}
This completes the proof.
\end{proof}

\begin{definition}\label{def-2-2}
A $d$-dimensional nucleus $\omega$ on $\mathcal{H}$ is a mapping $\omega \colon \mathbb{Z}^d \rightarrow \mathfrak{S}_+(\mathcal{H})$
satisfying that
\begin{equation}\label{eq-2-16}
  \sum_{\mathrm{x}\in \mathbb{Z}^d}\mathrm{Tr}[\omega(\mathrm{x})]=1.
\end{equation}
The set of all $d$-dimensional nucleuses on $\mathcal{H}$ is denoted as $\mathscr{T}^{(d)}(\mathcal{H})$.
\end{definition}

It can be seen that, for each $d$-dimensional nucleus $\omega \in \mathscr{T}^{(d)}(\mathcal{H})$,
the corresponding function $\mathrm{x}\mapsto \mathrm{Tr}[\omega(\mathrm{x})]$ defines
a probability distribution on $\mathbb{Z}^d$.

\begin{theorem}\label{thr-2-6}
For each $n\geq 0$, there exists a mapping
$\mathfrak{J}_n^{(d)} \colon \mathscr{T}^{(d)}(\mathcal{H}) \rightarrow \mathscr{T}^{(d)}(\mathcal{H})$ such that
\begin{equation}\label{eq-2-17}
  \big[\mathfrak{J}_n^{(d)}\omega\big](\mathrm{x})
  = \sum_{\varepsilon\in \Lambda^d}C_n^{(\varepsilon)}\omega(\mathrm{x}-\varepsilon)C_n^{(\varepsilon)},\quad \mathrm{x}\in \mathbb{Z}^d,\, \omega\in \mathscr{T}^{(d)}(\mathcal{H}),
\end{equation}
where $\sum_{\varepsilon\in \Lambda^d}$ means to sum over $\Lambda^d$.
\end{theorem}

\begin{proof}
Let $n\geq 0$.
For each $\omega\in \mathscr{T}^{(d)}(\mathcal{H})$, there is naturally a mapping $\omega'$ on $\mathbb{Z}^d$ associated with
$\omega$ in the following way
\begin{equation*}
  \omega'(\mathrm{x})=\sum_{\varepsilon\in \Lambda^d}C_n^{(\varepsilon)}\omega(\mathrm{x}-\varepsilon)C_n^{(\varepsilon)},\quad \mathrm{x}\in \mathbb{Z}^d.
\end{equation*}
We observe that $C_n^{(\varepsilon)}\omega(\mathrm{x}-\varepsilon)C_n^{(\varepsilon)}\in \mathfrak{S}_+(\mathcal{H})$
for all $\mathrm{x}\in \mathbb{Z}^d$ and all $\varepsilon\in \Lambda^d$, which implies that
$\omega'(\mathrm{x}) \in \mathfrak{S}_+(\mathcal{H})$ for all $\mathrm{x}\in \mathbb{Z}^d$,
hence $\omega'$ is a mapping from $\mathbb{Z}^d$ to $\mathfrak{S}_+(\mathcal{H})$.

Next we show that $\sum_{\mathrm{x}\in \mathbb{Z}^d}\mathrm{Tr}[\omega'(\mathrm{x})]=1$.
In fact, for each $\mathrm{x}\in \mathbb{Z}^d$, by Theorem~\ref{thr-2-5} we have
\begin{equation*}
  \sum_{\varepsilon\in \Lambda^d}\mathrm{Tr}\big[C_n^{(\varepsilon)}\omega(\mathrm{x})C_n^{(\varepsilon)}\big]
  = \mathrm{Tr}[\omega(\mathrm{x})].
\end{equation*}
Thus, in view of the fact that all the series involved have positive terms, we get
\begin{equation*}
  \sum_{\mathrm{x}\in \mathbb{Z}^d}\mathrm{Tr}[\omega'(\mathrm{x})]
  = \sum_{\mathrm{x}\in \mathbb{Z}^d}\sum_{\varepsilon\in \Lambda^d}
    \mathrm{Tr}\big[C_n^{(\varepsilon)}\omega(\mathrm{x}-\varepsilon)C_n^{(\varepsilon)}\big]
  = \sum_{\mathrm{x}\in \mathbb{Z}^d}\sum_{\varepsilon\in \Lambda^d}
    \mathrm{Tr}\big[C_n^{(\varepsilon)}\omega(\mathrm{x})C_n^{(\varepsilon)}\big]
  = \sum_{\mathrm{x}\in \mathbb{Z}^d}\mathrm{Tr}[\omega(\mathrm{x})],
\end{equation*}
which together with $\omega\in \mathscr{T}^{(d)}(\mathcal{H})$ implies that
$\sum_{\mathrm{x}\in \mathbb{Z}^d}\mathrm{Tr}[\omega'(\mathrm{x})]=1$. Now, according to Definition~\ref{def-2-2},
we know that $\omega'\in \mathscr{T}^{(d)}(\mathcal{H})$.
Finally, we define a mapping
$\mathfrak{J}_n^{(d)} \colon \mathscr{T}^{(d)}(\mathcal{H}) \rightarrow \mathscr{T}^{(d)}(\mathcal{H})$
as
\begin{equation*}
  \mathfrak{J}_n^{(d)}\omega = \omega',\quad \omega\in \mathscr{T}^{(d)}(\mathcal{H}).
\end{equation*}
Then $\mathfrak{J}_n^{(d)}$ is the desired.
\end{proof}

\subsection{Definition of walk}\label{subsec-2-3}

This subsection first offers the definition of our model of OQW and then examines its basic properties.

As mentioned above, we call a positive operator of trace class (on a Hilbert space)
a density operator if it has unit trace.
Recall that $d\geq 2$ is a given positive integer.
We denote by $l^2(\mathbb{Z}^d)$ the space of square summable complex-valued functions
on the $d$-dimensional integer lattice $\mathbb{Z}^d$. As a Hilbert space, $l^2(\mathbb{Z}^d)$ has
a countable orthonormal basis $\big\{\delta_{\mathrm{x}} \mid \mathrm{x}\in \mathbb{Z}^d\big\}$, which is known as
the canonical ONB of $l^2(\mathbb{Z}^d)$, where $\delta_{\mathrm{x}}$ is the function on $\mathbb{Z}^d$ given by
\begin{equation*}
  \delta_{\mathrm{x}}(\mathrm{z})=
  \left\{
    \begin{array}{ll}
      1, & \hbox{$\mathrm{z}=\mathrm{x}$, $\mathrm{z}\in \mathbb{Z}^d$;} \\
      0, & \hbox{$\mathrm{z}\neq \mathrm{x}$, $\mathrm{z}\in \mathbb{Z}^d$.}
    \end{array}
  \right.
\end{equation*}
By convention, we use $|\delta_{\mathrm{x}}\rangle\!\langle \delta_{\mathrm{x}}|$ to mean the Dirac operator associated with $\delta_{\mathrm{x}}$, which is a density operator on $l^2(\mathbb{Z}^d)$.

By the general theory of trace class operators on a Hilbert space \cite{simon}, one can easily come to the
next lemma, which provides a way to construct a density operator on
the tensor space $l^2(\mathbb{Z}^d)\otimes \mathcal{H}$ from the
canonical ONB of $l^2(\mathbb{Z}^d)$ and elements of $\mathscr{T}^{(d)}(\mathcal{H})$.

\begin{lemma}\label{lem-2-7}
Let $\omega \in \mathscr{T}^{(d)}(\mathcal{H})$. Then,  for each $\mathrm{x}\in \mathbb{Z}^d$, $|\delta_{\mathrm{x}}\rangle\!\langle\delta_{\mathrm{x}}|\otimes \omega(\mathrm{x})$
is a positive operator of trace class on $l^2(\mathbb{Z}^d)\otimes \mathcal{H}$. Moreover, the operator series
\begin{equation}\label{eq-2-17}
  \sum_{\mathrm{x}\in \mathbb{Z}^d}|\delta_{\mathrm{x}}\rangle\!\langle\delta_{\mathrm{x}}|\otimes \omega(\mathrm{x})
\end{equation}
is convergent in the trace operator norm and its sum operator is a density operator on $l^2(\mathbb{Z}^d)\otimes \mathcal{H}$.
\end{lemma}

With help of this lemma, we are now ready to introduce our model of OQW as follows.

\begin{definition}\label{def-2-3}
The $d$-dimensional open QBN walk is an open quantum walk on the $d$-dimensional integer lattice $\mathbb{Z}^d$ that admits the following features.
\begin{itemize}

\item Its states are represented by density operators on the tensor space $l^2(\mathbb{Z}^d)\otimes \mathcal{H}$.

\item Let $\widetilde{\omega^{(n)}}$ be the state of the walk at time $n\geq 0$. Then $\widetilde{\omega^{(n)}}$ takes the form
     \begin{equation}\label{eq-2-18}
     \widetilde{\omega^{(n)}} = \sum_{\mathrm{x} \in \mathbb{Z}^d}|\delta_{\mathrm{x}}\rangle\langle\delta_{\mathrm{x}}|\otimes \omega^{(n)}(\mathrm{x}),
     \end{equation}
  where $\omega^{(n)}\in \mathscr{T}^{(d)}(\mathcal{H})$, which is called the nucleus of the state $\widetilde{\omega^{(n)}}$.

  \item The time evolution of the walk is governed by equation
     \begin{equation}\label{eq-2-19}
      \omega^{(n+1)} = \mathfrak{J}_n^{(d)}\omega^{(n)},\quad n\geq 0,
      \end{equation}
 where $\omega^{(n+1)}$ and $\omega^{(n)}$ are the nucleuses of the states $\widetilde{\omega^{(n+1)}}$ and $\widetilde{\omega^{(n)}}$, respectively and $\mathfrak{J}_n^{(d)}$ is the mapping described in Theorem~\ref{thr-2-6}.
\end{itemize}
In that case, the function $\mathrm{x}\mapsto \mathrm{Tr}[\omega^{(n)}(\mathrm{x})]$ on $\mathbb{Z}^d$ is called
the probability distribution of the walk at time $n\geq 0$,
while the quantity $\mathrm{Tr}[\omega^{(n)}(\mathrm{x})]$ is the probability to find out the walker
at position $\mathrm{x}\in \mathbb{Z}^d$ and time $n\geq 0$.
By convention, the state $\widetilde{\omega^{(0)}}$ of the walk at time $n=0$ is usually known as its initial state.
\end{definition}

Physically, $l^2(\mathbb{Z}^d)$ describes the position of the walk, while $\mathcal{H}$ describes the
internal degrees of freedom of the walk. As shown above, $\mathcal{H}$ is infinitely dimensional,
which means that the $d$-dimensional open QBN walk has infinitely many internal degrees of freedom.

\begin{remark}
It is not hard to see that the $d$-dimensional open QBN walk is completely determined by the nucleus sequence of its states.
Let $\big(\omega^{(n)}\big)_{n\geq 0}$ be the nucleus sequence of states of the $d$-dimensional open QBN walk.
Then, by Theorem~\ref{thr-2-6}, one has the following evolution relations
\begin{equation}\label{eq-2-20}
  \omega^{(n+1)}(\mathrm{x}) = \sum_{\varepsilon\in \Lambda^d}C_n^{(\varepsilon)}\omega^{(n)}(\mathrm{x}-\varepsilon)C_n^{(\varepsilon)},
  \quad \mathrm{x}\in \mathbb{Z}^d,\, n\geq 0,
\end{equation}
which actually give an alternative description of the evolution of the $d$-dimensional open QBN walk.
\end{remark}

\subsection{Quantum channel representation}\label{subsec-2-4}

In this subsection, we mainly offer a quantum channel representation of the $d$-dimensional open QBN walk.

For $n\geq 0$ and $\mathrm{x}$, $\mathrm{y}\in \mathbb{Z}^d$, we define an operator $M^{(n)}(\mathrm{x},\mathrm{y})$ on the tensor space $l^2(\mathbb{Z}^d)\otimes \mathcal{H}$ as
\begin{equation}\label{eq-2-21}
  M^{(n)}_{\mathrm{x},\mathrm{y}}=
  \left\{
    \begin{array}{ll}
      |\delta_{\mathrm{x}}\rangle\!\langle \delta_{\mathrm{y}}|\otimes C_n^{(\mathrm{x}-\mathrm{y})}, & \hbox{$\mathrm{x}-\mathrm{y} \in \Lambda^d$;} \\
      0, & \hbox{$\mathrm{x}-\mathrm{y} \notin \Lambda^d$.}
    \end{array}
  \right.
\end{equation}
Clearly, for each $n \geq 0$, the operator family
$\big\{M^{(n)}_{\mathrm{x},\mathrm{y}} \mid \mathrm{x},\, \mathrm{y} \in \mathbb{Z}^d\big\}$
is infinite and countable.

\begin{theorem}\label{thr-2-8}
For each $n\geq 0$, the countable family
$\big\{M^{(n)}_{\mathrm{x},\mathrm{y}} \mid \mathrm{x},\, \mathrm{y} \in \mathbb{Z}^d\big\}$ of  operators satisfies
the following relation
\begin{equation}\label{eq-2-22}
  \sum_{\mathrm{x},\mathrm{y} \in \mathbb{Z}^d}{M^{(n)}_{\mathrm{x},\mathrm{y}}}^*M^{(n)}_{\mathrm{x},\mathrm{y}}=I,
\end{equation}
where ${M^{(n)}_{\mathrm{x},\mathrm{y}}}^*$ means the adjoint of $M^{(n)}_{\mathrm{x},\mathrm{y}}$, $I$ denotes the identity operator on $l^2(\mathbb{Z}^d)\otimes \mathcal{H}$
and the operator series converges strongly.
\end{theorem}

\begin{proof}
First, for each $\mathrm{y}\in \mathbb{Z}^d$, we find that there are at most finitely many $\mathrm{x}\in \mathbb{Z}^d$
such that $M^{(n)}_{\mathrm{x},\mathrm{y}}\neq 0$,
which implies that the series
$\sum_{\mathrm{x} \in \mathbb{Z}^d}{M^{(n)}_{\mathrm{x},\mathrm{y}}}^*M^{(n)}_{\mathrm{x},\mathrm{y}}$
is actually a sum of finitely many summands. On the other hand, for each $\mathrm{y}\in \mathbb{Z}^d$,
by a direct calculation we have
\begin{equation*}
\begin{split}
\sum_{\mathrm{x} \in \mathbb{Z}^d}{M^{(n)}_{\mathrm{x},\mathrm{y}}}^*M^{(n)}_{\mathrm{x},\mathrm{y}}
   &= \sum_{\mathrm{x}-\mathrm{y} \in \Lambda^d}
    \big(|\delta_{\mathrm{x}}\rangle\!\langle \delta_{\mathrm{y}}|\otimes C_n^{(\mathrm{x}-\mathrm{y})}\big)^*
    \big(|\delta_{\mathrm{x}}\rangle\!\langle \delta_{\mathrm{y}}|\otimes C_n^{(\mathrm{x}-\mathrm{y})}\big)\\
   &= \sum_{\mathrm{x}-\mathrm{y} \in \Lambda^d}
    \big(|\delta_{\mathrm{y}}\rangle\!\langle \delta_{\mathrm{x}}|\otimes C_n^{(\mathrm{x}-\mathrm{y})}\big)
    \big(|\delta_{\mathrm{x}}\rangle\!\langle \delta_{\mathrm{y}}|\otimes C_n^{(\mathrm{x}-\mathrm{y})}\big)\\
   &= \sum_{\mathrm{x}-\mathrm{y} \in \Lambda^d}|\delta_{\mathrm{y}}\rangle\!\langle \delta_{\mathrm{y}}|
      \otimes \big( C_n^{(\mathrm{x}-\mathrm{y})} C_n^{(\mathrm{x}-\mathrm{y})}\big)\\
   &= |\delta_{\mathrm{y}}\rangle\!\langle \delta_{\mathrm{y}}|\otimes \sum_{\mathrm{x}-\mathrm{y} \in \Lambda^d}
        C_n^{(\mathrm{x}-\mathrm{y})} C_n^{(\mathrm{x}-\mathrm{y})}\\
   &= |\delta_{\mathrm{y}}\rangle\!\langle \delta_{\mathrm{y}}|\otimes I_{\mathcal{H}}.
\end{split}
\end{equation*}
Therefore, using the fact that the series
$\sum_{\mathrm{y} \in \mathbb{Z}^d}|\delta_{\mathrm{y}}\rangle\!\langle \delta_{\mathrm{y}}|$
strongly converges to $I_{l^2(\mathbb{Z}^d)}$, we finally come to
\begin{equation*}
  \sum_{\mathrm{x},\mathrm{y} \in \mathbb{Z}^d}{M^{(n)}_{\mathrm{x},\mathrm{y}}}^*M^{(n)}_{\mathrm{x},\mathrm{y}}
  = \sum_{\mathrm{y} \in \mathbb{Z}^d}\sum_{\mathrm{x} \in \mathbb{Z}^d}
    {M^{(n)}_{\mathrm{x},\mathrm{y}}}^*M^{(n)}_{\mathrm{x},\mathrm{y}}
 = \sum_{\mathrm{y} \in \mathbb{Z}^d}|\delta_{\mathrm{y}}\rangle\!\langle \delta_{\mathrm{y}}|\otimes I_{\mathcal{H}}
 = I_{l^2(\mathbb{Z}^d)}\otimes I_{\mathcal{H}}
 = I.
\end{equation*}
Here, $I_{l^2(\mathbb{Z}^d)}$ and $I_{\mathcal{H}}$ mean the identity operators on $l^2(\mathbb{Z}^d)$
and $\mathcal{H}$, respectively.
\end{proof}

As above, we denote by $\mathfrak{S}\big(l^2(\mathbb{Z}^d)\otimes \mathcal{H}\big)$
the space of trace class operators on $l^2(\mathbb{Z}^d)\otimes \mathcal{H}$.
Then, by the general theory of trace class operators \cite{simon}, for each $n\geq 0$ and
each $\widetilde{\omega} \in \mathfrak{S}\big(l^2(\mathbb{Z}^d)\otimes \mathcal{H}\big)$,
the operator series
\begin{equation*}
  \sum_{\mathrm{x},\mathrm{y} \in \mathbb{Z}^d} M^{(n)}_{\mathrm{x},\mathrm{y}}\,\widetilde{\omega}\, {M^{(n)}_{\mathrm{x},\mathrm{y}}}^*
\end{equation*}
converges in the trace norm, and moreover its sum still belongs to
$\mathfrak{S}\big(l^2(\mathbb{Z}^d)\otimes \mathcal{H}\big)$.

\begin{definition}\label{def-2-4}
For $n\geq 0$, we define a mapping
$\mathsf{M}^{(n)}\colon \mathfrak{S}\big(l^2(\mathbb{Z}^d)\otimes \mathcal{H}\big) \rightarrow \mathfrak{S}\big(l^2(\mathbb{Z}^d)\otimes \mathcal{H}\big)$ as
\begin{equation}\label{eq}
  \mathsf{M}^{(n)}(\widetilde{\omega})
= \sum_{\mathrm{x},\mathrm{y} \in \mathbb{Z}^d} M^{(n)}_{\mathrm{x},\mathrm{y}}\, \widetilde{\omega}\, {M^{(n)}_{\mathrm{x},\mathrm{y}}}^*,\quad
\widetilde{\omega} \in \mathfrak{S}\big(l^2(\mathbb{Z}^d)\otimes \mathcal{H}\big),
\end{equation}
where $\sum_{\mathrm{x},\mathrm{y} \in \mathbb{Z}^d}$ means to sum for all $\mathrm{x}$, $\mathrm{y} \in \mathbb{Z}^d$.
\end{definition}

From a point of quantum information theory \cite{petz}, the mapping $\mathsf{M}^{(n)}$ is actually a quantum channel of Krauss type.
The next result then gives a quantum channel representation of the $d$-dimensional open QBN walk.

\begin{theorem}\label{thr-2-9}
Let $\left(\widetilde{\omega^{(n)}}\right)_{n\geq 0}$ the state sequence of the $d$-dimensional open QBN walk. Then it satisfies the following evolution equation
\begin{equation}\label{eq-2-24}
  \widetilde{\omega^{(n+1)}} = \mathsf{M}^{(n)}\left(\widetilde{\omega^{(n)}}\right),\quad n\geq 0.
\end{equation}
\end{theorem}

\begin{proof}
Let $n\geq 0$. By Definition~\ref{def-2-3}, $\widetilde{\omega^{(n)}}$ has a representation of the following form
\begin{equation*}
  \widetilde{\omega^{(n)}}
= \sum_{\mathrm{z} \in \mathbb{Z}^d}|\delta_{\mathrm{z}}\rangle\langle\delta_{\mathrm{z}}|\otimes \omega^{(n)}(\mathrm{z}),
\end{equation*}
where the series converges strongly. For $\mathrm{x}$, $\mathrm{y}\in \mathbb{Z}^d$
with $\mathrm{x}-\mathrm{y}\in \Lambda^d$, in view of the fact
\begin{equation*}
  |\delta_{\mathrm{x}}\rangle\langle\delta_{\mathrm{y}}|\,|\delta_{\mathrm{z}}\rangle\langle\delta_{\mathrm{z}}|
  =\left\{
     \begin{array}{ll}
       |\delta_{\mathrm{x}}\rangle\langle\delta_{\mathrm{y}}|, & \hbox{$\mathrm{z}=\mathrm{y}$;} \\
       0, & \hbox{$\mathrm{z}=\mathrm{y}$, $\mathrm{z}\in \mathbb{Z}^d$,}
     \end{array}
   \right.
\end{equation*}
we have
\begin{equation*}
\begin{split}
  M^{(n)}_{\mathrm{x},\mathrm{y}}\, \widetilde{\omega^{(n)}}\, {M^{(n)}_{\mathrm{x},\mathrm{y}}}^*
   &= \big[|\delta_{\mathrm{x}}\rangle\langle\delta_{\mathrm{y}}|\otimes C_n^{(\mathrm{x}-\mathrm{y})}\big]
      \Big[\sum_{\mathrm{z} \in \mathbb{Z}^d}|\delta_{\mathrm{z}}\rangle\langle\delta_{\mathrm{z}}|
      \otimes \omega^{(n)}(\mathrm{z})\Big]
      \big[|\delta_{\mathrm{x}}\rangle\langle\delta_{\mathrm{y}}|\otimes C_n^{(\mathrm{x}-\mathrm{y})}\big]^*\\
  &= \big[|\delta_{\mathrm{x}}\rangle\langle\delta_{\mathrm{y}}|
          \otimes \big(C_n^{(\mathrm{x}-\mathrm{y})}\omega^{(n)}(\mathrm{y})\big)\big]
      \big[|\delta_{\mathrm{x}}\rangle\langle\delta_{\mathrm{y}}|\otimes C_n^{(\mathrm{x}-\mathrm{y})}\big]^*\\
  &= \big[|\delta_{\mathrm{x}}\rangle\langle\delta_{\mathrm{y}}|
          \otimes \big(C_n^{(\mathrm{x}-\mathrm{y})}\omega^{(n)}(\mathrm{y})\big)\big]
      \big[|\delta_{\mathrm{y}}\rangle\langle\delta_{\mathrm{x}}|\otimes C_n^{(\mathrm{x}-\mathrm{y})}\big]\\
  &= |\delta_{\mathrm{x}}\rangle\langle\delta_{\mathrm{x}}| \otimes
      \big(C_n^{(\mathrm{x}-\mathrm{y})}\omega^{(n)}(\mathrm{y})C_n^{(\mathrm{x}-\mathrm{y})}\big).
\end{split}
\end{equation*}
For $\mathrm{x}$, $\mathrm{y}\in \mathbb{Z}^d$ with $\mathrm{x}-\mathrm{y}\in \Lambda^d$,
in view of $M^{(n)}_{\mathrm{x},\mathrm{y}}=0$, we simply have
$M^{(n)}_{\mathrm{x},\mathrm{y}}\, \widetilde{\omega^{(n)}}\, {M^{(n)}_{\mathrm{x},\mathrm{y}}}^*=0$.
Thus, for each $\mathrm{y}\in \mathbb{Z}^d$, as a sum actually with a finite number of summands,
\begin{equation*}
\begin{split}
  \sum_{\mathrm{x}\in \mathbb{Z}^d} M^{(n)}_{\mathrm{x},\mathrm{y}}\, \widetilde{\omega^{(n)}}\,
       {M^{(n)}_{\mathrm{x},\mathrm{y}}}^*
  &= \sum_{\mathrm{x}-\mathrm{y}\in \Lambda^d} M^{(n)}_{\mathrm{x},\mathrm{y}}\, \widetilde{\omega^{(n)}}\,
       {M^{(n)}_{\mathrm{x},\mathrm{y}}}^*\\
  &= \sum_{\mathrm{x}-\mathrm{y}\in \Lambda^d}|\delta_{\mathrm{x}}\rangle\langle\delta_{\mathrm{x}}| \otimes
      \big(C_n^{(\mathrm{x}-\mathrm{y})}\omega^{(n)}(\mathrm{y})C_n^{(\mathrm{x}-\mathrm{y})}\big)\\
  & = \sum_{\varepsilon\in \Lambda^d}|\delta_{\mathrm{y}+\varepsilon}\rangle\langle\delta_{\mathrm{y}+\varepsilon}| \otimes
      \big(C_n^{(\varepsilon)}\omega^{(n)}(\mathrm{y})C_n^{(\varepsilon)}\big).
\end{split}
\end{equation*}
Therefore
\begin{equation*}
\begin{split}
\mathsf{M}^{(n)}\left(\widetilde{\omega^{(n)}}\right)
& = \sum_{\mathrm{y} \in \mathbb{Z}^d} \sum_{\mathrm{x} \in \mathbb{Z}^d}
     M^{(n)}_{\mathrm{x},\mathrm{y}}\,\widetilde{\omega^{(n)}}\, {M^{(n)}_{\mathrm{x},\mathrm{y}}}^*\\
& = \sum_{\mathrm{y} \in \mathbb{Z}^d}
    \sum_{\varepsilon\in \Lambda^d}|\delta_{\mathrm{y}+\varepsilon}\rangle\langle\delta_{\mathrm{y}+\varepsilon}|
   \otimes \big(C_n^{(\varepsilon)}\omega^{(n)}(\mathrm{y})C_n^{(\varepsilon)}\big)\\
& = \sum_{\mathrm{x} \in \mathbb{Z}^d}|\delta_{\mathrm{x}}\rangle\langle\delta_{\mathrm{x}}|\otimes
    \sum_{\varepsilon\in \Lambda^d}
    C_n^{(\varepsilon)}\omega^{(n)}(\mathrm{x}-\varepsilon)C_n^{(\varepsilon)}\\
& = \sum_{\mathrm{x} \in \mathbb{Z}^d}|\delta_{\mathrm{x}}\rangle\langle\delta_{\mathrm{x}}|\otimes
    \omega^{(n+1)}(\mathrm{x})\\
& = \widetilde{\omega^{(n+1)}}.
\end{split}
\end{equation*}
Here we make use of the strong convergence of the involved series.
\end{proof}

\section{Probability distribution of walk}\label{sec-3}

In the present section, we consider the $d$-dimensional open QBN walk from a perspective
of probability distribution. We first show that the $d$-dimensional open QBN walk admits the
``separability-preserving'' property, and then, based on this property, we calculate explicitly
the limit probability distribution of the walk for a wide range of choices of its initial state.

\subsection{Separability-preserving property}\label{subsec-3-1}

Let $\mathscr{T}(\mathcal{H})$ be the $1$-dimensional counterpart of the $d$-dimensional nucleus set $\mathscr{T}^{(d)}(\mathcal{H})$, namely
\begin{equation}\label{eq-3-1}
\mathscr{T}(\mathcal{H})=\Big\{ \rho\colon \mathbb{Z} \rightarrow \mathfrak{S}_+(\mathcal{H})\Bigm| \sum_{x=-\infty}^{\infty}\mathrm{Tr}[\rho(x)]=1\Big\}.
\end{equation}
Elements of $\mathscr{T}(\mathcal{H})$ are called $1$-dimensional nucleuses on $\mathcal{H}$.

\begin{lemma}\label{lem-3-1}\cite{w-w-r-t}
For each $n\geq 0$, there exists a mapping $\mathfrak{J}_n \colon \mathscr{T}(\mathcal{H}) \rightarrow \mathscr{T}(\mathcal{H})$ such that
\begin{equation}\label{eq-3-2}
  [\mathfrak{J}_n\rho](x) =  L_n\rho(x+1)L_n + R_n\rho(x-1)R_n,\quad x\in \mathbb{Z},\, \rho\in \mathscr{T}(\mathcal{H}).
\end{equation}
\end{lemma}

By this Lemma, we find that $\left(\prod_{k=0}^n \mathfrak{J}_k\right)\rho \in \mathscr{T}(\mathcal{H})$ for all $n\geq 0$ whenever $\rho \in \mathscr{T}(\mathcal{H})$,
where $\prod_{k=0}^n \mathfrak{J}_k$ means the composition of mappings $\mathfrak{J}_0$, $\mathfrak{J}_1$, $\cdots$, $\mathfrak{J}_n$.

\begin{theorem}\label{thr-3-2}
Let $\rho_1$, $\rho_2$, $\cdots$, $\rho_d\in \mathscr{T}(\mathcal{H})$ be $1$-dimensional nucleuses on $\mathcal{H}$. Define
\begin{equation}\label{eq}
  \omega(\mathrm{x}) = \mathsf{K}\Big(\bigotimes_{j=1}^d \rho_j(x_j) \Big)\mathsf{K}^{-1},\quad \mathrm{x}=(x_1, x_2, \cdots, x_d)\in \mathbb{Z}^d.
\end{equation}
Then $\omega \in \mathscr{T}^{(d)}(\mathcal{H})$, namely $\omega$ is a $d$-dimensional nucleus on $\mathcal{H}$.
\end{theorem}

\begin{proof}
\begin{equation*}
\sum_{\mathrm{x}\in \mathbb{Z}^d}\mathrm{Tr}[\omega(\mathrm{x})]
= \sum_{\mathrm{x}\in \mathbb{Z}^d}\mathrm{Tr}\Big[\bigotimes_{j=1}^d \rho_j(x_j) \Big]
= \sum_{\mathrm{x}\in \mathbb{Z}^d} \prod_{j=1}^d\mathrm{Tr} [\rho_j(x_j)]
= \prod_{j=1}^d\sum_{x_j\in \mathbb{Z}}\mathrm{Tr} [\rho_j(x_j)]
= \prod_{j=1}^d 1 =1.
\end{equation*}
\end{proof}

\begin{definition}\label{def-3-1}
A $d$-dimensional nucleus $\omega \in \mathscr{T}^{(d)}(\mathcal{H})$ is said to
 to be separable if there exist $1$-dimensional nucleuses
$\rho_1$, $\rho_2$, $\cdots$, $\rho_d\in \mathscr{T}(\mathcal{H})$ such that
\begin{equation}\label{eq}
  \omega(\mathrm{x}) = \mathsf{K}\Big(\bigotimes_{j=1}^d \rho_j(x_j) \Big)\mathsf{K}^{-1},\quad \mathrm{x}=(x_1, x_2, \cdots, x_d)\in \mathbb{Z}^d.
\end{equation}
A state $\widetilde{\omega^{(n)}}$ of the $d$-dimensional open QBN walk is said to be separable
if its nucleus $\omega^{(n)}$ is separable.
\end{definition}

The next theorem shows that all the states of the $d$-dimensional open QBN walk are separable provided
its initial state is separable. In other words, the $d$-dimensional open QBN walk
has the ``separability-preserving'' property.

\begin{theorem}\label{thr-3-3}
Let $\big(\omega^{(n)}\big)_{n\geq 0}$ be the nucleus sequence of the states of the $d$-dimensional open QBN walk. Suppose that
\begin{equation}\label{eq-3-5}
  \omega^{(0)}(\mathrm{x}) = \mathsf{K}\Big(\bigotimes_{j=1}^d \rho_j^{(0)}(x_j) \Big)\mathsf{K}^{-1},\quad \mathrm{x}=(x_1, x_2, \cdots, x_d)\in \mathbb{Z}^d,
\end{equation}
where $\rho_1^{(0)}$, $\rho_2^{(0)}$, $\cdots$, $\rho_d^{(0)}\in \mathscr{T}(\mathcal{H})$. Then, for all $n\geq 1$, $\omega^{(n)}$ has a representation of the following form
\begin{equation}\label{eq-3-6}
  \omega^{(n)}(\mathrm{x}) = \mathsf{K}\Big(\bigotimes_{j=1}^d \rho_j^{(n)}(x_j) \Big)\mathsf{K}^{-1},\quad \mathrm{x}=(x_1, x_2, \cdots, x_d)\in \mathbb{Z}^d,
\end{equation}
where
\begin{equation}\label{eq-3-7}
\rho_j^{(n)} = \Big(\prod_{k=0}^{n-1}\mathfrak{J}_k\Big)\rho_j^{(0)}
\end{equation}
for $j=1$, $2$, $\cdots$, $d$.
\end{theorem}

\begin{proof}
According to Lemma~\ref{lem-3-1}, $\rho_j^{(n)} \in \mathscr{T}(\mathcal{H})$ for each $j$ with $1\leq j \leq d$
and each $n\geq 0$. Thus, by Theorem~\ref{thr-3-2}, there exists a sequence $\big(\omega'^{(n)}\big)_{n\geq 0}$
in $\mathscr{T}^{(d)}(\mathcal{H})$ such that
\begin{equation}\label{eq-3-8}
  \omega'^{(n)}(\mathrm{x}) = \mathsf{K}\Big(\bigotimes_{j=1}^d \rho_j^{(n)}(x_j) \Big)\mathsf{K}^{-1},\quad
  \mathrm{x}=(x_1, x_2, \cdots, x_d)\in \mathbb{Z}^d,\ \ n\geq 0.
\end{equation}
In particular, we have
\begin{equation*}
  \omega'^{(0)}(\mathrm{x}) = \mathsf{K}\Big(\bigotimes_{j=1}^d \rho_j^{(0)}(x_j) \Big)\mathsf{K}^{-1},\quad
  \mathrm{x}=(x_1, x_2, \cdots, x_d)\in \mathbb{Z}^d,
\end{equation*}
which, together with the assumption given in (\ref{eq-3-5}), implies that $\omega^{(0)}=\omega'^{(0)}$.

On the other hand, for $n\geq 1$, by (\ref{eq-3-7}) we have $\rho_j^{(n)}= \mathfrak{J}_{n-1}\rho_j^{(n-1)}$,
$1\leq j\leq d$,
which together with Lemma~\ref{lem-3-1} implies that
\begin{equation*}
  \rho_j^{(n)}(x_j)
  = \sum_{\varepsilon_j\in \Lambda}B_{n-1}^{(\varepsilon_j)}\rho_j^{(n-1)}(x_j-\varepsilon_j)B_{n-1}^{(\varepsilon_j)},\quad
  x_j \in \mathbb{Z},\ \ 1\leq j \leq d.
\end{equation*}
Here $B_{n-1}^{(-1)}=L_{n-1}$ and $B_{n-1}^{(+1)}=R_{n-1}$ as indicated in Definition~\ref{def-2-1}.
Taking tensor product gives
\begin{equation*}
 \bigotimes_{j=1}^d \rho_j^{(n)}(x_j)
  = \sum_{\varepsilon\in \Lambda^d}
     \Big(\bigotimes_{j=1}^dB_{n-1}^{(\varepsilon_j)}\Big)
     \Big(\bigotimes_{j=1}^d\rho_j^{(n-1)}(x_j-\varepsilon_j)\Big)
     \Big(\bigotimes_{j=1}^dB_{n-1}^{(\varepsilon_j)}\Big),
\end{equation*}
$\mathrm{x}=(x_1, x_2, \cdots, x_d) \in \mathbb{Z}^d$,\ $n\geq 1$,
where $\varepsilon=(\varepsilon_1, \varepsilon_2, \cdots, \varepsilon_d)$.
This, together with (\ref{eq-3-8}) and Definition~\ref{def-2-1}, yields
\begin{equation*}
 \omega'^{(n)}(\mathrm{x})
  = \sum_{\varepsilon\in \Lambda^d}
     C_{n-1}^{(\varepsilon)}
     \omega'^{(n-1)}(\mathrm{x}-\varepsilon)
     C_{n-1}^{(\varepsilon)},\quad \mathrm{x}=(x_1, x_2, \cdots, x_d) \in \mathbb{Z}^d,\ n\geq 1,
\end{equation*}
namely
\begin{equation*}
 \omega'^{(n)}
  = \mathfrak{J}_{n-1}^{(d)}\omega'^{(n-1)},\quad n\geq 1,
\end{equation*}
which together with (\ref{eq-2-19}) and $\omega^{(0)}=\omega'^{(0)}$ implies that $\omega^{(n)}=\omega'^{(n)}$ for all
$n\geq 0$, which together with (\ref{eq-3-8}) gives (\ref{eq-3-6}).
\end{proof}

As an immediate consequence of the previous theorem, we have the following corollary,
which gives a formula for calculating the probability distributions of the $d$-dimensional open QBN walk.

\begin{corollary}\label{coro-3-4}
Let the nucleus $\omega^{(0)}$ of the initial state of the $d$-dimensional open QBN walk take the following form
\begin{equation}\label{eq-3-9}
  \omega^{(0)}(\mathrm{x}) = \mathsf{K}\Big(\bigotimes_{j=1}^d \rho_j^{(0)}(x_j) \Big)\mathsf{K}^{-1},\quad \mathrm{x}=(x_1, x_2, \cdots, x_d)\in \mathbb{Z}^d,
\end{equation}
where $\rho_1^{(0)}$, $\rho_2^{(0)}$, $\cdots$, $\rho_d^{(0)}\in \mathscr{T}(\mathcal{H})$. Then, at time $n\geq 1$, the walk has a probability distribution of the following form
\begin{equation}\label{eq-3-10}
  \mathrm{Tr}\big[\omega^{(n)}(\mathrm{x})\big] = \prod_{j=1}^d \mathrm{Tr}\big[\rho_j^{(n)}(x_j)\big],\quad \mathrm{x}=(x_1, x_2,\cdots, x_d)\in \mathbb{Z}^d,
\end{equation}
where
\begin{equation}\label{eq-3-11}
\rho_j^{(n)} = \Big(\prod_{k=0}^{n-1}\mathfrak{J}_k\Big)\rho_j^{(0)}
\end{equation}
for $j=1$, $2$, $\cdots$, $d$.
\end{corollary}

\subsection{Limit probability distribution}\label{subsec-3-2}

Let $\rho$ be a $1$-dimensional nucleus on $\mathcal{H}$, namely $\rho\in \mathscr{T}(\mathcal{H})$.
A sequence $\big(\rho^{(n)}\big)_{n\geq 0}$ of $1$-dimensional nucleuses on $\mathcal{H}$ is said to be generated by $\rho$
and $\big(\mathfrak{J}_n\big)_{n\geq 0}$ if $\rho^{(0)}=\rho$ and
\begin{equation*}
\rho^{(n+1)} = \mathfrak{J}_n\rho^{(n)},\quad n\geq 0.
\end{equation*}
We note that if a sequence $\big(\rho^{(n)}\big)_{n\geq 0}$ is generated by $\rho$
and $\big(\mathfrak{J}_n\big)_{n\geq 0}$, then $\rho^{(0)}=\rho$ and
\begin{equation*}
\rho^{(n)} = \Big(\prod_{k=0}^{n-1}\mathfrak{J}_k\Big)\rho^{(0)},\quad n\geq 1,
\end{equation*}
where $\prod_{k=0}^{n-1}\mathfrak{J}_k$ means the composition of mappings $\{\,\mathfrak{J}_k \mid 0\leq k \leq n-1\,\}$.

\begin{definition}\label{def-3-2}
A $1$-dimensional nucleus $\rho$ on $\mathcal{H}$ is said to be regular if
the sequence $\big(\rho^{(n)}\big)_{n\geq 0}$ generated by $\rho$
and $\big(\mathfrak{J}_n\big)_{n\geq 0}$ satisfies that
\begin{equation*}
  \lim_{n\to \infty}\sum_{x\in \mathbb{Z}}\mathrm{e}^{\frac{\mathrm{i}tx}{\sqrt{n}}}\mathrm{Tr}\big[\rho^{(n)}(x)\big]
  = \mathrm{e}^{-\frac{t^2}{2}},\quad t\in \mathbb{R}.
\end{equation*}
\end{definition}

The next example shows that there exist infinitely many $1$-dimensional nucleuses on $\mathcal{H}$ that are regular.

\begin{example}\label{exam-3-1}
Let $\sigma\in \Gamma$ and $Z_{\sigma}$ the corresponding basis vector of the canonical ONB of $\mathcal{H}$.
Define
\begin{equation}\label{eq-3-12}
  \rho_{\sigma}(x)=\left\{
            \begin{array}{ll}
              |Z_{\sigma}\rangle\!\langle Z_{\sigma}|, & \hbox{$x=0$;} \\
              0, & \hbox{$x\neq 0$, $x\in \mathbb{Z}$,}
            \end{array}
          \right.
\end{equation}
where $|Z_{\sigma}\rangle\!\langle Z_{\sigma}|$ is the Dirac operator associated with the basis vector $Z_{\sigma}$.
Then $\rho_{\sigma}$ is a regular $1$-dimensional nucleus on $\mathcal{H}$.
\end{example}

\begin{proof}
Clearly $\rho_{\sigma}$ is a $1$-dimensional nucleus on $\mathcal{H}$. Next we show that it is also regular.
To this end, we consider the space $l^2(\mathbb{Z},\mathcal{H})$ of square summable $\mathcal{H}$-valued functions defined on $\mathbb{Z}$,
which is endowed the usual inner product and norm. It can be verified (see \cite{w-r-t} and references therein) that for each $n\geq 0$,
there exists a unitary operator $\mathcal{U}_n$ on $l^2(\mathbb{Z}, \mathcal{H})$ such that
\begin{equation}\label{eq-3-13}
  (\mathcal{U}_n\Phi)(x) = R_n\Phi(x-1) + L_n\Phi(x+1),\quad x\in \mathbb{Z},\; \Phi\in l^2(\mathbb{Z}, \mathcal{H}).
\end{equation}
Now define $\Phi_n= \big(\prod_{k=0}^{n-1}\mathcal{U}_k\big)\Phi_0$, $n\geq 1$, where $\Phi_0\in l^2(\mathbb{Z}, \mathcal{H})$
is taken as
\begin{equation*}
  \Phi_0(x)=\left\{
              \begin{array}{ll}
                Z_{\sigma}, & \hbox{$x=0$;} \\
                0, & \hbox{$x\neq 0$, $x\in \mathbb{Z}$.}
              \end{array}
            \right.
\end{equation*}
Then, by a result recently proven by Wang \textit{et al} (see \cite{w-r-t} for details), we have
\begin{equation}\label{eq-3-14}
\|\Phi_n(x)\|^2=
 \left\{
   \begin{array}{ll}
     \frac{1}{2^n}\binom{n}{j}, & \hbox{$x=n-2j$, $0\leq j \leq n$;} \\
     0, & \hbox{otherwise,}
   \end{array}
 \right.
\end{equation}
for all $n\geq 1$. On the other hand, let $\big(\rho^{(n)}\big)_{n\geq 0}$ be the $1$-dimensional nucleus sequence on
$\mathcal{H}$ generated by $\rho_{\sigma}$ and $\big(\mathfrak{J}_n\big)_{n\geq 0}$.
Then, by a careful check, we find that
\begin{equation*}
  \rho^{(0)}(x) = \rho_{\sigma}(x)=|\Phi_0(x)\rangle\langle \Phi_0(x)|,\quad x\in \mathbb{Z}.
\end{equation*}
This, together with Theorem 4.2 of \cite{w-w-r-t}, as well as (\ref{eq-3-14}), implies that
\begin{equation}\label{eq-4-15}
\mathrm{Tr}[\rho^{(n)}(x)]=
 \left\{
   \begin{array}{ll}
     \frac{1}{2^n}\binom{n}{j}, & \hbox{$x=n-2j$, $0\leq j \leq n$;} \\
     0, & \hbox{otherwise.}
   \end{array}
 \right.
\end{equation}
Consequently, by a careful calculation, we finally come to
\begin{equation*}
  \lim_{n\to \infty}\sum_{x\in \mathbb{Z}}\mathrm{e}^{\frac{\mathrm{i}tx}{\sqrt{n}}}\mathrm{Tr}\big[\rho^{(n)}(x)\big]
  = \lim_{n\to \infty}\cos^n\frac{t}{\sqrt{n}}
  =\mathrm{e}^{-\frac{t^2}{2}},\quad t\in \mathbb{R},
\end{equation*}
which means that $\rho_{\sigma}$ is regular.
\end{proof}

The next theorem shows that, for a wide range of choices of its initial state, the $d$-dimensional open QBN walk has
a limit probability distribution of $d$-dimensional Gauss type.

\begin{theorem}\label{thr-3-5}
Let $\big(\widetilde{\omega^{(n)}}\big)_{n\geq 0}$ be the state sequence of the $d$-dimensional open QBN walk
and $\omega^{(n)}$ the nucleus of $\widetilde{\omega^{(n)}}$, $n\geq 0$. Suppose that
the nucleus $\omega^{(0)}$ of the initial state $\widetilde{\omega^{(0)}}$ takes the following form
\begin{equation}\label{eq-3-16}
  \omega^{(0)}(\mathrm{x}) = \mathsf{K}\Big(\bigotimes_{j=1}^d \rho_j^{(0)}(x_j) \Big)\mathsf{K}^{-1},\quad \mathrm{x}=(x_1, x_2, \cdots, x_d)\in \mathbb{Z}^d,
\end{equation}
where $\big\{\,\rho_j^{(0)} \mid 1\leq j \leq d \,\big\}$ are regular $1$-dimensional nucleuses on $\mathcal{H}$.
For $n\geq 0$, let $X_n$ be a $d$-dimensional random vector with probability distribution
\begin{equation}\label{eq-3-17}
  P\{X_n = \mathrm{x}\} = \mathrm{Tr}[\omega^{(n)}(\mathrm{x})],\quad \mathrm{x}\in \mathbb{Z}^d.
\end{equation}
Then
\begin{equation*}
\frac{X_n}{\sqrt{n}}\Longrightarrow N(\mathbf{0},I_{d\times d}),
\end{equation*}
namely $\frac{X_n}{\sqrt{n}}$ converges in law to
the $d$-dimensional standard Gauss distribution as $n\rightarrow \infty$.
\end{theorem}

\begin{proof}
Let $n\geq 1$.
Consider the characteristic function $C_{\frac{X_n}{\sqrt{n}}}(\mathbf{t})$ of the random vector $\frac{X_n}{\sqrt{n}}$.
By definition, we have
\begin{equation}\label{eq-3-18}
  C_{\frac{X_n}{\sqrt{n}}}(\mathbf{t})
= \sum_{\mathbf{x}\in \mathbb{Z}^d}\mathrm{e}^{\frac{\mathrm{i}}{\sqrt{n}}\sum_{j=1}^dt_jx_j}
\mathrm{Tr}[\omega^{(n)}(\mathrm{x})],\quad \mathbf{t}=(t_1, t_2,\cdots, t_d)\in \mathbb{R}^d,
\end{equation}
where $\mathbf{x}=(x_1, x_2,\cdots, x_d)$. Using Corollary~\ref{coro-3-4} gives
\begin{equation*}
  C_{\frac{X_n}{\sqrt{n}}}(\mathbf{t})
= \prod_{j=1}^d \Big(\sum_{x_j\in \mathbb{Z}}\mathrm{e}^{\frac{\mathrm{i}t_jx_j}{\sqrt{n}}}
  \mathrm{Tr}\big[\rho_j^{(n)}(x_j)\big]\Big), \quad \mathbf{t}=(t_1, t_2,\cdots, t_d)\in \mathbb{R}^d,
\end{equation*}
where
\begin{equation*}
\rho_j^{(n)} = \Big(\prod_{k=0}^{n-1}\mathfrak{J}_k\Big)\rho_j^{(0)},\quad 1\leq j\leq d.
\end{equation*}
Now consider $\lim_{n\to \infty}C_{\frac{X_n}{\sqrt{n}}}(\mathbf{t})$.
For each $j$, since $\rho_j^{(0)}$ is regular and $\big(\rho_j^{(n)}\big)_{n\geq 0}$ is generated by $\rho_j^{(0)}$ and $\big(\mathfrak{J}_n\big)_{n\geq 0}$, we have
\begin{equation*}
  \lim_{n\to \infty}\sum_{x_j\in \mathbb{Z}}\mathrm{e}^{\frac{\mathrm{i}t_jx_j}{\sqrt{n}}} \mathrm{Tr}\big[\rho_j^{(n)}(x_j)\big]
   = e^{-\frac{t_j^2}{2}},\quad t_j \in \mathbb{R},
\end{equation*}
which implies that
\begin{equation*}
  \lim_{n\to \infty}C_{\frac{X_n}{\sqrt{n}}}(\mathbf{t})
= \prod_{j=1}^d \Big(\lim_{n\to \infty}\sum_{x_j\in \mathbb{Z}}\mathrm{e}^{\frac{\mathrm{i}t_jx_j}{\sqrt{n}}}
  \mathrm{Tr}\big[\rho_j^{(n)}(x_j)\big]\Big)
= \prod_{j=1}^d e^{-\frac{t_j^2}{2}}
= e^{-\frac{1}{2} \sum_{j=1}^d t_j^2},
\end{equation*}
$\mathbf{t}=(t_1, t_2,\cdots, t_d)\in \mathbb{R}^d$. Thus, $\frac{X_n}{\sqrt{n}}$ converges in law to
the $d$-dimensional standard Gauss distribution as $n\rightarrow \infty$.
\end{proof}

\section{Links with unitary quantum walk}\label{sec-4}

In the final section, we unveil links between the $d$-dimensional open QBN walk
and the unitary quantum walk recently introduced in \cite{w-w}, which was called the $d$-dimensional QBN walk therein.

As before, $d\geq 2$ is a given positive integer.
We denote by $l^2\big(\mathbb{Z}^d, \mathcal{H}\big)$ the space of square summable functions
defined on $\mathbb{Z}^d$ and valued in $\mathcal{H}$, namely
\begin{equation}\label{eq-4-1}
  l^2\big(\mathbb{Z}^d, \mathcal{H}\big)
= \Big\{ W\colon \mathbb{Z}^d\rightarrow \mathcal{H} \Bigm| \sum_{\mathrm{x}\in \mathbb{Z}^d} \|W(\mathrm{x})\|^2<\infty\Big\},
\end{equation}
where $\|\cdot\|$ is the norm in $\mathcal{H}$.
It is known that $l^2\big(\mathbb{Z}^d, \mathcal{H}\big)$ forms a separable Hilbert space with
the inner product induced by the that in $\mathcal{H}$.

Recall that, for each $n\geq 0$,  $\big\{C_n^{(\varepsilon)}\mid \varepsilon\in \Lambda^d\big\}$ are self-adjoint
operators on $\mathcal{H}$ with properties that:\ $C_n^{(\varepsilon)}C_n^{(\varepsilon')}=0$
for $\varepsilon$, $\varepsilon'\in \Lambda^d$ with $\varepsilon\neq \varepsilon'$; and their sum
$\sum_{\varepsilon\in \Lambda^d}C_n^{(\varepsilon)}$ is a unitary operator on $\mathcal{H}$
(see Subsection~\ref{subsec-2-2} for details).
Using these facts, we can prove that there exists a sequence of unitary operators $\big(\mathcal{U}_n^{(d)}\big)_{n\geq 0}$
on $l^2\big(\mathbb{Z}^d, \mathcal{H}\big)$ such that
\begin{equation}\label{eq-4-2}
  \big(\mathcal{U}_n^{(d)}W\big)(\mathbf{x}) =\sum_{\varepsilon \in \Lambda^d} C_n^{(\varepsilon)}W(\mathbf{x}-\varepsilon),\quad \mathbf{x}\in \mathbb{Z}^d,\; W \in l^2\big(\mathbb{Z}^d, \mathcal{H}\big),\, n\geq 0.
\end{equation}
With these unitary operators as the evolution operators, the authors of \cite{w-w} introduced a unitary quantum walk on the $d$-dimensional integer lattice $\mathbb{Z}^d$
in the following manner.

\begin{definition}\label{def-4-1}\cite{w-w}
The $d$-dimensional QBN walk is a discrete-time unitary quantum walk on the $d$-dimensional integer lattice $\mathbb{Z}^d$ that satisfies the following requirements.
\begin{itemize}
\item Its states are represented by unit vectors in space $l^2\big(\mathbb{Z}^d, \mathcal{H}\big)$.

\item The time evolution of the walk is governed by equation
     \begin{equation}\label{eq-4-3}
      W_{n+1}= \mathcal{U}_n^{(d)}W_n,\quad  n\geq 0,
      \end{equation}
 where $W_n \in l^2\big(\mathbb{Z}^d, \mathcal{H}\big)$ denotes the state of the walk at time $n\geq 0$, and in particular $W_0$ is the initial state of the walk.
\end{itemize}
In that case, the function $\mathbf{x}\mapsto \|W_n(\mathbf{x})\|^2$ on $\mathbb{Z}^d$ is called the probability distribution of the walk at time $n\geq 0$,
while the quantity $\|W_n(\mathbf{x})\|^2$ is the probability to find out the walker
at position $\mathbf{x}\in \mathbb{Z}^d$ and time $n\geq 0$.
\end{definition}

\begin{remark}\label{rem-4-1}
According to (\ref{eq-4-2}), which describes the definition of unitary operators $\mathcal{U}_n^{(d)}$,
the evolution equation of the $d$-dimensional QBN walk can also be represented as
\begin{equation}\label{eq-4-4}
  W_{n+1}(\mathrm{x})
  = \sum_{\varepsilon\in \Lambda^d} C_n^{(\varepsilon)}W_n(\mathrm{x}-\varepsilon),\quad \mathrm{x} \in \mathbb{Z}^d,\, n\geq 0,
\end{equation}
which is more convenient to use.
\end{remark}

As is seen, the $d$-dimensional QBN walk is driven by the sequence $\big(\mathcal{U}_n^{(d)}\big)_{n\geq 0}$
of unitary operators. Hence it belongs to the category of unitary quantum walks. In other words,
it is indeed a unitary quantum walk.
The next result shows links between the $d$-dimensional open QBN walk and the $d$-dimensional QBN walk.

\begin{theorem}\label{thr-4-1}
Let $\big(\widetilde{\omega^{(n)}}\big)_{n\geq 0}$ be the state sequence of the $d$-dimensional open QBN walk,
where $\omega^{(n)}$ is the nucleus of $\widetilde{\omega^{(n)}}$.
Let $\big(W_n\big)_{n\geq 0}$ be the state sequence of the $d$-dimensional QBN walk.
Suppose that
\begin{equation}\label{eq-4-5}
  \omega^{(0)}(\mathrm{x}) = |W_0(\mathrm{x})\rangle\langle W_0(\mathrm{x})|,\quad \mathrm{x}\in \mathbb{Z}^d,
\end{equation}
where $|W_0(\mathrm{x})\rangle\langle W_0(\mathrm{x})|$ is the Dirac operator associated with $W_0(\mathrm{x})$.
Then, for all $n\geq 0$, it holds that
\begin{equation}\label{eq-4-6}
  \mathrm{Tr}\big[\omega^{(n)}(\mathrm{x})\big] = \|W_n(\mathrm{x})\|^2,\quad \mathrm{x}\in \mathbb{Z}^d.
\end{equation}
\end{theorem}

\begin{proof}
We first recall some algebraic and analytical properties of operators $C_n^{(\varepsilon)}$.
As is indicated above, for each $n\geq 0$, $\big\{C_n^{(\varepsilon)}\mid \varepsilon\in \Lambda^d\big\}$
are self-adjoint operators on $\mathcal{H}$ with the property that $C_n^{(\varepsilon)}C_n^{(\varepsilon')}=0$
for $\varepsilon$, $\varepsilon'\in \Lambda^d$ with $\varepsilon\neq \varepsilon'$, which implies that
\begin{equation}\label{eq-4-7}
  \Big\|\sum_{\varepsilon\in \Lambda^d}C_n^{(\varepsilon)}\xi_{\varepsilon}\Big\|^2
  = \sum_{\varepsilon\in \Lambda^d}\big\|C_n^{(\varepsilon)}\xi_{\varepsilon}\big\|^2
\end{equation}
whenever $\big\{\,\xi_{\varepsilon} \mid \varepsilon\in \Lambda^d\,\big\} \subset \mathcal{H}$.
And moreover, it follows from (\ref{eq-2-10}) and Definition~\ref{def-2-1}
that all the operators $\big\{\,C_n^{(\varepsilon)} \mid \varepsilon\in \Lambda^d,\, n\geq 0\,\big\}$ form a commutative family.

Now we consider (\ref{eq-4-6}). Clearly, it holds for $n=0$. In the following, we let $n\geq 1$ and $\mathrm{x}\in \mathbb{Z}^d$.
Using (\ref{eq-4-4}) and (\ref{eq-4-7}) gives
\begin{equation*}
  \|W_n(\mathrm{x})\|^2
    =\sum_{\varepsilon^{(n-1)}\in \Lambda^d}\big\|C_{n-1}^{(\varepsilon^{(n-1)})}W_{n-1}\big(\mathrm{x}-\varepsilon^{(n-1)}\big)\big\|^2,
\end{equation*}
where, by using the commutativity of the family
$\big\{\,C_k^{(\varepsilon)} \mid \varepsilon\in \Lambda^d,\, k\geq 0\,\big\}$
and (\ref{eq-4-7}), we have
\begin{equation*}
  \big\|C_{n-1}^{(\varepsilon^{(n-1)})}W_{n-1}\big(\mathrm{x}-\varepsilon^{(n-1)}\big)\big\|^2
  = \sum_{\varepsilon^{(n-2)}\in \Lambda^d}\big\|C_{n-1}^{(\varepsilon^{(n-1)})}
    C_{n-2}^{(\varepsilon^{(n-2)})}W_{n-2}\big(\mathrm{x}-\varepsilon^{(n-1)}-\varepsilon^{(n-2)}\big)\big\|^2.
\end{equation*}
Thus
\begin{equation*}
 \|W_n(\mathrm{x})\|^2
 = \sum_{\varepsilon^{(n-1)}\in \Lambda^d}\sum_{\varepsilon^{(n-2)}\in \Lambda^d}\big\|C_{n-1}^{(\varepsilon^{(n-1)})}
    C_{n-2}^{(\varepsilon^{(n-2)})}W_{n-2}\big(\mathrm{x}-\varepsilon^{(n-1)}-\varepsilon^{(n-2)}\big)\big\|^2.
\end{equation*}
It then follows by the induction that
\begin{equation}\label{eq-4-8}
 \|W_n(\mathrm{x})\|^2
 = \sum_{\varepsilon^{(n-1)}\in \Lambda^d}\cdots \sum_{\varepsilon^{(0)}\in \Lambda^d}
 \big\|C_{n-1}^{(\varepsilon^{(n-1)})}\cdots C_{0}^{(\varepsilon^{(0)})}
 W_{0}\big(\mathrm{x}-\varepsilon^{(n-1)}- \cdots - \varepsilon^{(0)}\big)\big\|^2.
\end{equation}
Similarly, repeatedly using the evolution relation (\ref{eq-2-20}) of the $d$-dimensional open QBN walk yields that
\begin{equation*}
  \omega^{(n)}(\mathrm{x})
  = \sum_{\varepsilon^{(n-1)}\in \Lambda^d}\cdots \sum_{\varepsilon^{(0)}\in \Lambda^d}
    C_{n-1}^{(\varepsilon^{(n-1)})}\cdots C_{0}^{(\varepsilon^{(0)})}
    \omega^{(0)}\big(\mathrm{x}-\varepsilon^{(n-1)}-\cdots - \varepsilon^{(0)}\big)
    C_{0}^{(\varepsilon^{(0)})}\cdots C_{n-1}^{(\varepsilon^{(n-1)})}.
\end{equation*}
Taking the trace gives
\begin{equation}\label{eq-4-9}
\begin{split}
 \mathrm{Tr}&\big[ \omega^{(n)}(\mathrm{x})\big]\\
  &= \sum_{\varepsilon^{(n-1)}\in \Lambda^d}\cdots \sum_{\varepsilon^{(0)}\in \Lambda^d}
   \mathrm{Tr}\big[ C_{n-1}^{(\varepsilon^{(n-1)})}\cdots C_{0}^{(\varepsilon^{(0)})}
    \omega^{(0)}\big(\mathrm{x}-\varepsilon^{(n-1)}-\cdots - \varepsilon^{(0)}\big)
    C_{0}^{(\varepsilon^{(0)})}\cdots C_{n-1}^{(\varepsilon^{(n-1)})}\big].
\end{split}
\end{equation}
On the other hand,  for $\varepsilon^{(n-1)}\in \Lambda^d$, $\cdots$, $\varepsilon^{(0)}\in \Lambda^d$, by using the assumption (\ref{eq-4-5}) we find
\begin{equation*}
\begin{split}
&\big\|C_{n-1}^{(\varepsilon^{(n-1)})}\cdots C_{0}^{(\varepsilon^{(0)})}
          W_{0}\big(\mathrm{x}-\varepsilon^{(n-1)}- \cdots - \varepsilon^{(0)}\big)\big\|^2\\
  &= \mathrm{Tr}\big[\big| C_{n-1}^{(\varepsilon^{(n-1)})}\cdots C_{0}^{(\varepsilon^{(0)})} W_{0}\big(\mathrm{x}-\varepsilon^{(n-1)}- \cdots - \varepsilon^{(0)}\big)\big\rangle \big\langle C_{n-1}^{(\varepsilon^{(n-1)})}\cdots C_{0}^{(\varepsilon^{(0)})} W_{0}\big(\mathrm{x}-\varepsilon^{(n-1)}- \cdots - \varepsilon^{(0)}\big)\big|\big]\\
  &= \mathrm{Tr}\big[ C_{n-1}^{(\varepsilon^{(n-1)})}\cdots C_{0}^{(\varepsilon^{(0)})}
           \omega^{(0)}\big(\mathrm{x}-\varepsilon^{(n-1)}-\cdots - \varepsilon^{(0)}\big)
          C_{0}^{(\varepsilon^{(0)})}\cdots C_{n-1}^{(\varepsilon^{(n-1)})}\big],
\end{split}
\end{equation*}
which, together with (\ref{eq-4-8}) and (\ref{eq-4-9}), implies that $\mathrm{Tr}\big[ \omega^{(n)}(\mathrm{x})\big] = \|W_n(\mathrm{x})\|^2$.
\end{proof}

\begin{remark}\label{rem-4-2}
It should be mentioned that under the conditions given in Theorem~\ref{thr-4-1}, one has, in general, that
\begin{equation}\label{eq-4-10}
  \omega^{(n)}(\mathrm{x}) \neq |W_n(\mathrm{x})\rangle\langle W_n(\mathrm{x})|,\quad \mathrm{x}\in \mathbb{Z}^d,\, n\geq 1.
\end{equation}
This suggests that the $d$-dimensional open QBN walk is mathematically different from the
$d$-dimensional QBN walk.
\end{remark}

However, from a perspective of physical realization, transition may happen between
these two walks. As an immediate consequence of Theorem~\ref{thr-4-1}, the next corollary
describes the limit case when such transition happens.

\begin{corollary}\label{coro-4-2}
Let the initial state $\widetilde{\omega^{(0)}}$ of the $d$-dimensional open QBN walk be such that
\begin{equation}\label{eq-4-11}
  \widetilde{\omega^{(0)}} = \sum_{\mathrm{x}\in \mathbb{Z}^d}|\delta_{\mathrm{x}}\rangle\langle \delta_{\mathrm{x}}|
  \otimes |W_0(\mathrm{x})\rangle\langle W_0(\mathrm{x})|,
\end{equation}
where $W_0$ is the initial state of the $d$-dimensional QBN walk. Then, the $d$-dimensional open QBN walk has a limit probability distribution
if and only if the $d$-dimensional QBN walk
has a limit probability distribution. In that case, their limit probability distributions
are identical.
\end{corollary}

\section*{Acknowledgement}

This work is partly supported by National Natural Science Foundation of China (Grant No. 11861057).

\end{document}